\documentclass[A4paper,12pt]{article}
\usepackage{latexsym}
\usepackage{mathrsfs}
\usepackage{amssymb}
\usepackage{amscd}
\usepackage[dvips]{graphicx}                  
\newtheorem{definition}{Definition}
\newtheorem{theorem}{Theorem}

\newtheorem{lemma}{Lemma}

\newenvironment{proof}[1][Proof]{\textbf{#1.} }{\ \rule{0.5em}{0.5em}}
\long\def\symbolfootnote[#1]#2{\begingroup%
	\def\thefootnote{$\;$}\footnote[#1]{$^*$#2}\endgroup}
\begin{document}
	
	\title{Addendum to On Kuratowski partitions in the Marczewski and Laver structures and Ellentuck topology}
	\author{ Joanna Jureczko\footnote{The author is partially supported by Wroc\l{}aw Univercity of Science and Technology grant of K34W04D03 no. 8201003902.}}
	
		\maketitle
	
	\symbolfootnote[2]{Mathematics Subject Classification (2010): 54C30, 03E05, 03E40, 28A20.  
		
		\hspace{0.2cm}
		}
	
	\begin{abstract} In this paper we present the generalizations of results, given in the paper published in Georgian J. Math. 26(2019) no. 4, pp 591-598, towards point-finite families. 
	\end{abstract}
	
	\section{Introduction}
	The ain of this paper is to continue investigations given in \cite{FJW} concerning the old problem posed by Kuratowski, \cite{KK1}, towards point-finite families. The results presented here can be important for further investigations in this direction, because the topic around Kuratowski problem is still alive. The point-finite version of our previous results is motivated by  \cite{BCGR}.
	
	From the formal point of view the main results presented here are more general than those given in \cite{FJW}, since a partition is a special case of a point-finite cover. From the technical point of view, the proofs in both results are very similar, because they are also based on Fusion Lemma, (see \cite{AT, JB, TJ1}) which cannot be omitted in proofs, but there are some nuances, thus we show here them in proofs.
	
	The paper is organized as follows. In Section 2 there are given definitions and previous results, in Section 3 there are given main results.
	
	We use the standard terminology for the field. The definitions and facts not cited here  one can be find in \cite{RE, KK} (topology), \cite{EL} (Ellentuck topology), \cite{TJ} (set theory), \cite{JB, TJ1, JMS} (forcing).

		\section{Definitions and previous results}
	
	\subsection{Tree ideals}
	
	Let $K\subseteq \omega$ be a set, (finite or infinite).
	A set $T\subseteq K^{<\omega}$ is called a \textit{tree} iff $t\upharpoonright n \in T$ for all $t \in T$ and $n\in \omega$.
	Let $\mathbb{T}$ means a family of all tress. 
	For each $T \in \mathbb{T}$ and $t \in T$ the set 
	$$split(t, T) = |\{n \in K \colon t^\smallfrown n \in T\}|$$ denotes the number of successors of nodes in $T$.

	\begin{definition}
		A tree $T$ is called 
		\begin{enumerate}
			\item Sacks or perfect tree, denoted $T \in \mathbb{S}$, iff $K=\{0,1\}$ and $split(t, T) = 2$ for each $t \in T$,
			\item Laver tree, denoted $T \in \mathbb{L}$, iff $K=\omega$ and $split(t, T)$ is infinite for each $t \in T$.  
		\end{enumerate}
	\end{definition}
	
	Let $$[T] = \{x \in K^\omega \colon \forall_{n \in \omega}\ x \upharpoonright n \in T \}$$
	be the set of all infinite paths through $T$.
	\\
	Notice that $[T]$ is closed in the Baire space $K^\omega$, (see e.g. \cite{TJ}).
	
	By $stem(T)$ we mean a node $t \in T$ such that $split(s, T) = 1$ and $split(t, T) >~1$ for any $s \varsubsetneq t$.
	
	The ordering on $\mathbb{S}$ is defined as follows
	$Q\leqslant T $ iff $ Q \subseteq T$
	and 
	$$Q \leqslant_n T \textrm{ iff } Q \leqslant T \textrm{ and  any node of $n$-level of $T$ is a node of $n$-level of $Q$}.$$
	\indent
	If $T \in \mathbb{L}$, then $\{x \in [T] \colon stem(T) \in x\}$, (i.e. the part of $T$ above the $stem(T)$),  can be enumerated as follows:
	$s^T_0 = stem(T), s^T_1, ..., s^T_n, ...$\ .
	Thus we can define the ordering on $\mathbb{L}$ in the following way: let $Q, T \in \mathbb{L}$, $Q \leqslant T$ iff $Q \subseteq T$
	and
	$$Q \leqslant_n T \textrm{ iff } stem(Q) = stem(T) \textrm{ and } s^Q_i = s^T_i \textrm{ for all } i = 0, ..., n.$$

	We say that a set $A \subseteq K^\omega$ is a \textit{$t$-set} iff 
	$$\forall_{T \in \mathbb{T}}\ \exists_{Q \in \mathbb{T}}\ Q \subseteq T \wedge ([Q] \subseteq A \vee [Q]\cap A =\emptyset). $$
	We say that a set $A \subseteq 2^\omega$ is a \textit{$(t^0)$-set} iff 
	$$\forall_{T \in \mathbb{T}}\ \exists_{Q \in \mathbb{T}}\ Q \subseteq T \wedge  [Q]\cap A =\emptyset. $$
	For $\mathbb{S}$ we use the notation $(s)$- and $(s^0)$-sets for $(t)$- and $(t^0)$-sets, respectively, but for  $\mathbb{L}$ we use the notation $(l)$- and $(l^0)$-sets for $(t)$- and $(t^0)$-sets, respectively.
	
	Notice that all $(s^0)$-sets ($(l^0)$-sets) form a $\sigma$-ideal in $\mathbb{S}$ (in $\mathbb{L}$)  which we denote by $\mathbb{S}^0$ (by $\mathbb{L}^0$). 
	For further considerations, unless otherwise stated, $\mathbb{T}$ and $\mathbb{T}^0$ mean $\sigma$-ideals of: Sacks trees, (i.e. $(s)$- and $(s_0)$-tree, respectively) and Laver tree (i.e. $(l)$- and $(l^0)$-tree,respectively). Then $t, t^0$ and $K$ will be determined accordingly to theses structures.

	\subsection{Ellentuck topology}
	
	The Ellentuck topology $[\omega]^{\omega}_{EL}$ on $[\omega]^\omega$ is generated by sets of the form
	$$[a, A] = \{B \in [A]^\omega \colon a \subset B \subseteq a \cup A\},$$
	where $a \in [\omega]^{<\omega}$ and $A \in [\omega]^\omega$. We call such sets \textit{Ellentuck sets, (shortly $EL$-sets).} Obviously $[a, A] \subseteq [b, B]$ iff $b \subseteq a$ and $A \subseteq B$. 
	
	A set $M \subseteq [\omega]^\omega$ is \textit{completely Ramsey}, (shortly \textit{$CR$-set}), if for every $[a, A]$ there exists $B \in [A]^\omega$ such that $[a, B] \subseteq M$ or $[a, B] \cap M = \emptyset.$ 
	A set $M \subseteq [\omega]^\omega$ is \textit{nowhere Ramsey}, (shortly \textit{$NR$-set}), if for every $[a, A]$ there exists $B \in [A]^\omega$ such that $[a, B] \cap M = \emptyset.$
	\\
	Notice that all $NR$-sets form a $\sigma$-ideal in $[\omega]^{\omega}_{EL}$ which we denote by $\mathbb{NR}$. 
	
		\subsection{Fusion Lemma}
	Let $\mathbb{T}$  be the family of all trees.  
	A sequence $\{T_n\}_{n \in \omega}$ of trees such that 
	$$T_0 \geqslant_0 T_1 \geqslant_1 ... \geqslant_{n-1} T_n \geqslant_n ...$$
	is called a \textit{fusion sequence}.
	\\
	\\
	\textbf{Fact 1 (\cite{TJ1})} If $\{T_n\}_{n\in \omega}$ is a fusion sequence then $T = \bigcap_{n\in \omega}T_n$, (the fusion of $\{T_n\}_{n \in \omega}$), belongs to $\mathbb{T}$.
	\\ 
	\\
	A sequence $\{[a_n, A_n]\}_{n \in \omega}$ of $EL$-sets  is called a \textit{fusion sequence} if it is infinite and
	\\(1) $\{a_n\}_{n \in \omega}$ is a nondecreasing sequence of integers converging to infinity,
	\\(2) $A_{n+1} \in [a_n, A_n]$ for all $n \in \omega$.
	\\
	\\
	\textbf{Fact 2 (\cite{TJ1})} If $\{[a_n, A_n]\}_{n \in \omega}$ is a fusion sequence then $$[a, A] = \bigcap_{n\in \omega}[a_, A_n] =  [\bigcap_{n\in \omega} a_n,\bigcap_{n\in \omega}  A_n],$$ (the fusion of $\{[a_n, A_n]\}_{n \in \omega})$, is an $EL$-set.
	
	\section{Main results}
	
	In this part we present the main results of this paper. Recall that a family $\mathcal{F}$ of subsets of a topological space $X$ is called \textit{a point-finite family} if for $x \in X$ the family $\{A \in \mathcal{F} \colon x \in A\}$ is finite.
	
	We start with auxiliary lemmas which proofs are very similar to proofs of \cite[Lemmas 4.1-4.3]{FJW}. Thus we only indicate the differences.
	
		\begin{lemma}
			\begin{enumerate}
		\item Let $A \in P(K^\omega) \setminus \mathbb{T}^0$. For any  point-finite cover $\mathcal{F}$ of $A$ consisiting of  $t^0$-sets and for any perfect tree $T\in \mathbb{S}$ there exists a perfect subtree $Q \leqslant T$  such that the family
		$$\mathcal{F}_{[Q]} = \{F \cap [Q] \colon F_\alpha \in \mathcal{F}\}$$ has cardinality continuum.
		\item 	Let $M \in P([\omega]^\omega)\setminus \mathbb{NR}$ be an open and dense set. For any point-finite cover $\mathcal{F}$ of $M$ consisting of $NR$-sets and for any $[a, A] \subseteq [\omega]^{\omega}_{EL} $ there exists $[b, B] \subseteq [a, A]$ such that the family $$\mathcal{F}_{[b, B]} = \{F \cap [b, B] \colon F \in \mathcal{F}\}$$ has cardinality continuum.
		\end{enumerate}
	\end{lemma}
	
	\begin{proof}
			For our convenience we will show the proof of the first part of lemma. The second part is similar. 
		
		Let $\mathcal{F}$ be a point-finite cover of $A \in P(K^\omega) \setminus \mathbb{S}^0$ consisting of $t^0$-sets and
		let $T \in \mathbb{T}$. We will construct inductively by $n \in \omega$ a collection of subfamilies $\{\mathcal{F}_h \colon h\in k^n\}$, ($k \in \omega$, $k=2$ for Sacks trees and $k=n$ for Laver trees, compare \cite[Lemma 4.1-4.2]{FJW}) of $\mathcal{F}$ and a collection of perfect subtrees $\{T_h \colon h \in k^n\}$ of $T$ with the following properties:
		for any distinct $h, h' \in k^n$ 
		\begin{itemize}
	
	\item[(i)] $\mathcal{F}_h \subseteq \mathcal{F}$ and $ T_h \leqslant T$;
		
	\item[(ii)] $\bigcup\{\mathcal{F}_{h} \colon h \in k^n \} = \bigcup \mathcal{F}$;
		
	\item[(iii)] $\bigcup\mathcal{F}_h \not \in \mathbb{T}^0$;
		
	\item[(iv)] $\mathcal{F}_{h} \subseteq \mathcal{F}_{g}$ and $ T_h \leqslant_n T_{g} $, i.e. $[T_h] \subseteq [T_{g}]$, for $h\cap g = g$;
		
	\item[(v)] $\mathcal{F}_{h} \cap \mathcal{F}_{h'} = \emptyset$ and $[T_h]\cap [T_{h'}] =\emptyset$;
		
	\item[(vi)] $A \cap [T_h] \subseteq \mathcal{F}_h$ and $A\cap [T_h] \not \in \mathbb{T}^0$.
	\end{itemize}	
		
		Assume that for some $m \in \omega$ we have constructed the families $\{\mathcal{F}_h \colon h \in k^m \}$  and $\{T_h \colon h \in k^m\}$ of properties (i) - (vi).
		
		Now fix $h \in k^m$ and split $\mathcal{F}_h$ into disjoint  families $\mathcal{F}_{h^{\smallfrown} 0}, \mathcal{F}_{h^{\smallfrown} 1} = \mathcal{F} \setminus \mathcal{F}_{h^{\smallfrown} 0}$ such that
		$\bigcup\mathcal{F}_{h^{\smallfrown} 0},$ $ \bigcup\mathcal{F}_{h^{\smallfrown} 1} \not \in \mathbb{T}^0$.
		(Such splitting is possible because of (iii)).
		Now we will construct $T_{h^{\smallfrown} 0}, T_{h^{\smallfrown} 1} \leqslant_m T_h$ of properties (iv) and (vi). The inductive step flows as is done in proofs of \cite[Lemma 4.1-4.2]{FJW}.
		
		Now take $Q = \bigcap_{n \in \omega} \bigcup_{h \in 2^n} T_h$.
		By Fact 1 the tree  $Q\in \mathbb{T}$. Then $\mathcal{F}_{[Q]}$ has the required property. 
	\end{proof}

	\begin{theorem}
		
		\begin{enumerate}
			\item 
			Let $A \in P(K^\omega)\setminus \mathbb{T}^0$ and let $\mathcal{F}$ be a point-finite cover of $A$ consisting of $t^0$-sets. If for each $T \in \mathbb{T}$, with $A \cap [T] \not = \emptyset$, there exists a subtree $Q \leqslant T$, (with $A \cap [Q] \not = \emptyset$),  such that $\mathcal{F}_{[Q]} = \{F \cap [Q] \colon F \in \mathcal{F}\}$ has cardinality continuum then 	$\bigcup \mathcal{F'}_{[Q]}$ is not a $t$-set for some subfamily $\mathcal{F}' \subseteq \mathcal{F}$.
			\item 	Let $M \in P([\omega]^{\omega})\setminus \mathbb{NR}$ and let $\mathcal{F}$ be a point-finite cover of $A$ consisting of $NR$-sets. If for each $EL$-set $[a, A]$, with $M \cap [a, A] \not = \emptyset$, there exists an $EL$-set $[b, B] \leqslant [a, A]$, (with $M \cap [b, B] \not = \emptyset$),  such that $\mathcal{F}_{[b, B]} = \{F \cap [b, B] \colon F \in \mathcal{F}\}$ has cardinality continuum then 	$\bigcup \mathcal{F'}_{[b, B]}$ is not a $CR$-set for some subfamily $\mathcal{F}' \subseteq \mathcal{F}$.
		\end{enumerate}
	\end{theorem}
	
	\begin{proof}
		For our convenience we will show the proof of the first part of theorem. The second part is similar.
		
		Enumerate 
		$\mathbb{T}' = \{T_\alpha \in \mathbb{T} \colon A\cap [T_\alpha] \not = \emptyset, \alpha \in 2^\omega\}$.
		Let $\mathcal{F}$ be a  point-finite cover of $A$ consisting of $t^0$-sets. By assumption, for each  tree $T_\alpha \in \mathbb{T}'$ there exists a subtree $Q_\alpha \leqslant T_\alpha$ such that the point-finite $\mathcal{F}_{[Q_\alpha]} = \{F \cap [Q_\alpha] \colon F \in \mathcal{F}\}$ has cardinality continuum.
		Hence, for each $\alpha \in 2^\omega$ we will choose distinct elements $x_\alpha, y_\alpha \in [Q_\alpha]$,  such that the families
		$$\mathcal{B}^0_\alpha = \{F \in \mathcal{F} \colon x_\alpha \in F \cap [Q_\alpha]\} \setminus \{F \in \mathcal{F} \colon F \in \{\mathcal{B}^0_\beta\cup \mathcal{B}^1_{\beta}\colon \beta < \alpha\}\}$$
			$$\mathcal{B}^1_\alpha = \{F \in \mathcal{F} \colon x_\alpha \in F \cap [Q_\alpha]\} \setminus \{F \in \mathcal{F} \colon F \in \{\mathcal{B}^0_\beta\cup \mathcal{B}^1_{\beta}\cup \mathcal{B}^0_\alpha\colon \beta < \alpha\}\}$$
		are non-empty.
		(Such choice is possible because  $\mathcal{F}_{[Q_\alpha]}$ has cardinality continuum, but each $\mathcal{F}^0_\beta$ and $\mathcal{F}^1_\beta$ are finite and $\beta<\alpha$).
		Notice that $\mathcal{B}^0_\alpha$ and $\mathcal{B}^1_\alpha$ are disjoint. 
		
		Now, let $\mathcal{B}^0 = \{\mathcal{B}^0_\alpha \colon \alpha < 2^\omega\}$ and $\mathcal{B}^1 = \{\mathcal{B}^1_\alpha \colon \alpha < 2^\omega\}$ are disjoint. 
			
		Notice that $\bigcup \mathcal{B}^{\varepsilon}$ are not $t$-sets for any $\varepsilon \in \{0, 1\}$.
		Indeed. Suppose that  $\bigcup \mathcal{B}^{\varepsilon}$ is a $t$-set for some $\varepsilon \in \{0, 1\}$. Then there exists $Q_\alpha \leqslant T_\alpha$ such that $[Q_\alpha] \cap \bigcup \mathcal{B}^{\varepsilon} = \emptyset$. But by the construction we have that 
		$\{F \in \mathcal{F} \colon F \cap [Q_\alpha] \cap \bigcup \mathcal{B}^{\varepsilon} \not = \emptyset\}$ is non-empty. A contradiction.
		
		If $\bigcup \mathcal{B}^{\varepsilon}$ is not a $t$-set for some $\varepsilon \in \{0, 1\}$, then there exists $Q_\alpha \leqslant T_\alpha$ such that $[Q_\alpha] \subseteq \bigcup \mathcal{B}^{\varepsilon}$ and by the construction $[Q_\alpha] \cap \bigcup \mathcal{B}^{1 - \varepsilon} \not = \emptyset$ which contradicts with disjointness of families $\mathcal{B}^{0}$ and  $\mathcal{B}^{1}$. 
	\end{proof}

		\noindent
	{\sc Joanna Jureczko}
	\\
	Wroc\l{}aw University of Science and Technology,
	Faculty of Electronics, Wroc\l{}aw, Poland
	\\
	{\sl e-mail: joanna.jureczko@pwr.edu.pl}

\begin{thebibliography}{9}
		
		\thispagestyle{empty}
		
			\bibitem{AT}  S. Argyros, S. Todorcevic S., 
		\emph{Ramsey Methods in Analysis},
		advanced Courses in Mathematics, CRM Barcelona, 2005.
		
		\bibitem {JB} J. E. Baumgartner, 
		Iterated forcing, in: \emph{Surveys in se theory} (Ed. A. R. D. Mathias),
		London Math. Soc. Lecture Notes Series, 87, Cambridge University Press 1983, 1--59.
		
		\bibitem{BCGR} J. Brzuchowski, J. Cicho\'n, E. Grzegorek,  C. Ryll-Nardzewski, On the existence of nonmeasurable unions. Bull. Acad. Polon. Sci. S\'er. Sci. Math. 27 (1979), no. 6, 447--448.
		
		\bibitem {EL} E. Ellentuck, 
		A new proof that analytic sets are Ramsey,
		\emph{J. Symb. Log.} 39 (1974), 163--165.
		
		\bibitem{RE}  R. Engelking,
		\emph{General Topology}, (Revised and completed edition),
		Heldermann Verlag Berlin,  1989.
		
	    \bibitem{FJW} R. Frankiewicz, J. Jureczko, B. Weglorz, On Kuratowski partitions in the Marczewski and Laver structures and Ellentuck topology, Georgian J. Math., 26 (2019), no. 4, 591-598.
		
		\bibitem {TJ1} T. Jech, 
		\emph{Multiple forcing},
		Cambridge Tracts in Mathematics, 88. Cambridge University Press, Cambridge, 1986.
		
		\bibitem{TJ} T. Jech,
		\emph{Set theory},
		The third millennium edition, revised and expanded. Springer Monographs in Mathematics. Springer-Verlag, Berlin, 2003.
		
		\bibitem{JMS} H. Judah, A. W Miller and S. Shelah,
		Sacks foring, Laver forcing, and Mathias axiom, 
		\emph{Arch. Math. Logic} 31 (1992), 145--161.
		
		\bibitem {KK} K. Kuratowski, 
		\emph{Topology} vol 1,
		Polish Scientific Publ. ; New York ; London : Academic Press, 1966.
		
		\bibitem {KK1} K. Kuratowski,
		Quelques problemes concernant les espaces m\'etriques nons\'eparables,
		\emph{Fund. Math.} 25 (1935), 534--545. 
	
	\end{thebibliography}
\end{document}